\documentclass{birkjour}
\usepackage{amsfonts,amsmath,amsthm,amscd,amssymb,latexsym,cite,verbatim,texdraw,floatflt,pb-diagram}

\usepackage{tikz} % для рисунков
 \newtheorem{thm}{Theorem}[section]
 \newtheorem{cor}[thm]{Corollary}
 
 \newtheorem{prop}[thm]{Proposition}
 \theoremstyle{definition}
 \newtheorem{defn}[thm]{Definition}
 \theoremstyle{remark}
 \newtheorem{rem}[thm]{Remark}
 \newtheorem*{ex}{Example}
 \numberwithin{equation}{section}

\newcommand{\ve}{\varepsilon}

\begin{document}

%-------------------------------------------------------------------------
% editorial commands: to be inserted by the editorial office
%
%\firstpage{1} \volume{228} \Copyrightyear{2004} \DOI{003-0001}
%\seriesextraline{This is the Concrete Title of this Book\br H.E. R and S.T.C. W, Eds.}
%\firstpage{1}

\title[Mappings contracting perimeters of triangles]{Fixed point theorem for mappings \\ {contracting} perimeters of triangles}

%----------Author 1
\author[Evgeniy Petrov]{Evgeniy Petrov}

\address{
Institute of Applied Mathematics and Mechanics\\
of the NAS of Ukraine\\
Batiuka str. 19\\
84116 Slovyansk\\
Ukraine}

\email{eugeniy.petrov@gmail.com}

%\thanks{This work was completed with the support of our \TeX-pert.}
%----------Author 2
%\author{A Second Author}
%\address{The address of\br
%the second author\br
%sitting somewhere\br
%in the world}
%\email{dont@know.who.knows}


\date{July 30, 2022}

\begin{abstract}
We consider a new type of mappings in metric spaces which can be characterized as mappings contracting perimeters of triangles. It is shown that such mappings are continuous. The fixed point theorem for such mappings is proved and the classical Banach fixed-point theorem is obtained like a simple corollary. An example of a mapping contractive perimeters of triangles which is not a contraction mapping is constructed.
\end{abstract}

%%% ----------------------------------------------------------------------
\maketitle
%%% ----------------------------------------------------------------------
%\tableofcontents

\section{Introduction}
The Contraction Mapping Principle was established by S. Banach in his dissertation (1920) and published in 1922~\cite{Ba22}. Although the idea of successive approximations in a number of concrete situations (solution of differential and integral equations, approximation theory) had appeared earlier in the works of P.~L.~Chebyshev, E.~Picard, R. Caccioppoli, and others, S. Banach was the first who formulated this result in a correct abstract form suitable for a wide range of applications. After a century, the interest of mathematicians around the world to fixed-point theorems is still very high.
This is confirmed by the appearance in recent decades of numerous articles and monographs devoted to the fixed point theory and its applications, see, e.g., the monographs~\cite{Ki01,AJS18,Su18} for a survey on fixed point results.

The Banach contraction principle has been generalized in many ways over the years. In~\cite{KS14} authors noted that except Banach's fixed point theorem there are also three classical fixed point theorems against which metric extensions are usually checked. These are, respectively, Nadler's well-known set-valued extension of Banach's theorem~\cite{Na69}, the extension of Banach's theorem to nonexpansive mappings~\cite{Ki65}, and Caristi's theorem~\cite{Ca76}. At the same time it is possible to distinguish at least two types of generalizations of such theorems: in the first case the contractive nature of the mapping is weakened, see, e.g.~\cite{BW69,Ci74,Ki03,Mk69,Ra62,Re72,Su06,Wa12,Pr20,Po21}; in the second case the topology is weakened, see, e.g.~\cite{DBC12,Ba00,CS12,DD07,Fr00,JKR11,KKR90,KS13,LS12,Sa10,SRR09,Ta74,Tu12,SIIR20}.

Let $X$ be a metric spaces. In the present paper we consider a new type of mappings $T\colon X\to X$ which can be characterized as mappings contracting perimeters of triangles and prove the fixed point theorem for such mappings. Although the proof of the main theorem of this work is based on the ideas of the proof of Banach's classical theorem, the essential difference is that the definition of our mappings is based on the mapping of three points of the space instead of two. Moreover, we additionally require necessary condition which prevents the mapping $T$ from having points with least period two, $T(T(x))\neq x$ for all $x\in X$ such that  $Tx\neq x$. The ordinary contraction mappings form an important subclass of these mappings which immediately allows us to obtain the classical Banach's theorem like a simple corollary. An example of a mapping contracting perimeters of triangles which is not a contraction mapping is constructed for a space $X$ with $|X|=\aleph_0$, where $|X|$ is the cardinality of the set $X$.

\section{Mappings contracting perimeters of triangles}

\begin{defn}\label{d1}
Let $(X,d)$ be a metric space with $|X|\geqslant 3$. We shall say that $T\colon X\to X$ is a \emph{mapping contracting perimeters of triangles} on $X$ if there exists $\alpha\in [0,1)$ such that the inequality
  \begin{equation}\label{e1}
   d(Tx,Ty)+d(Ty,Tz)+d(Tx,Tz) \leqslant \alpha (d(x,y)+d(y,z)+d(x,z))
  \end{equation}
  holds for all three pairwise distinct points $x,y,z \in X$.
\end{defn}

\begin{rem}
Note that the requirement for $x,y,z\in X$ to be pairwise distinct is essential. One can see that otherwise this definition is equivalent to the definition of contraction mapping.
\end{rem}

\begin{prop}
Mappings contracting perimeters of triangles are continuous.
\end{prop}

\begin{proof}
Let $(X,d)$ be a metric space with $|X|\geqslant 3$, $T\colon X\to X$ be a mapping contracting perimeters of triangles on $X$ and let $x_0$ be an isolated point in $X$. Then, clearly, $T$ is continuous at $x_0$. Let now $x_0$ be an accumulation point. Let us show that for every $\ve>0$, there exists $\delta>0$ such that $d(Tx_0,Tx)<\ve$ whenever $d(x_0,x)<\delta$.
Since $x_0$ is an accumulation point, for every $\delta>0$ there exists $y\in X$ such that $d(x_0,y)< \delta$. By~(\ref{e1}) we have
$$
d(Tx_0,Tx)\leqslant d(Tx_0,Tx)+d(Tx_0,Ty)+d(Tx,Ty)
$$
$$
\leqslant \alpha(d(x_0,x)+d(x_0,y)+d(x,y)).
$$
Using the triangle inequality $d(x,y) \leqslant d(x_0,x)+d(x_0,y)$, we have
$$
d(Tx_0,Tx)\leqslant
2\alpha(d(x_0,x)+d(x_0,y))<2\alpha(\delta + \delta)=4\alpha\delta.
$$
Setting $\delta=\ve /(4\alpha)$, we obtain the desired inequality.
\end{proof}

\begin{thm}\label{t1}
Let $(X,d)$, $|X|\geqslant 3$, be a complete metric space and let the mapping $T\colon X\to X$ satisfy the following two conditions:
\begin{itemize}
  \item [(i)] $T(T(x))\neq x$ for all $x\in X$  such that  $Tx\neq x$.
  \item [(ii)] $T$ is a mapping contracting perimeters of triangles on $X$.
\end{itemize}
Then $T$ has a fixed point. The number of fixed points is at most two.
\end{thm}

\begin{proof}
Let $x_0\in X$, $Tx_0=x_1$, $Tx_1=x_2$, \ldots, $Tx_n=x_{n+1}$, \ldots. Suppose that $x_i$ is not a fixed point of the mapping $T$ for every $i=0,1,...$. Let us show that all $x_i$ are different. Since $x_i$ is not fixed, then $x_i\neq x_{i+1}=Tx_i$. By condition (i) $x_{i+2}=T(T(x_i))\neq x_i$ and by the supposition that $x_{i+1}$ is not fixed we have $x_{i+1}\neq x_{i+2}=Tx_{i+1}$. Hence, $x_i$, $x_{i+1}$ and $x_{i+2}$ are pairwise distinct. Further, set
$$
p_0=d(x_0,x_1)+d(x_1,x_2)+d(x_2,x_0),
$$
$$
p_1=d(x_1,x_2)+d(x_2,x_3)+d(x_3,x_1),
$$
$$
\cdots
$$
$$
p_n=d(x_n,x_{n+1})+d(x_{n+1},x_{n+2})+d(x_{n+2},x_n),
$$
$$
\cdots .
$$
Since  $x_i$, $x_{i+1}$ and $x_{i+2}$ are pairwise distinct by~(\ref{e1}) we have $p_1\leqslant \alpha p_0$, $p_2\leqslant \alpha p_1$, \ldots, $p_n\leqslant \alpha p_{n-1}$ and
\begin{equation}\label{e2}
p_0>p_1>...>p_n>\ldots .
\end{equation}
Suppose that $j\geqslant 3$ is a minimal natural number such that $x_j=x_i$ for some $i$ such that $0\leqslant i<j-2$. Then  $x_{j+1}=x_{i+1}$, $x_{j+2}=x_{i+2}$. Hence, $p_i=p_j$ which contradicts to~(\ref{e2}).

Further, let us show that $\{x_i\}$ is a Cauchy sequence.  It is clear that
$$
d(x_1,x_2)\leqslant p_0,
$$

$$
d(x_2,x_3)\leqslant p_1\leqslant \alpha p_0,
$$

$$
d(x_3,x_4)\leqslant p_2\leqslant \alpha p_1\leqslant  \alpha^2 p_0,
$$
$$
\cdots
$$
$$
d(x_n,x_{n+1})\leqslant p_{n-1}\leqslant \alpha^{n-1} p_0,
$$
$$
d(x_{n+1},x_{n+2})\leqslant p_{n}\leqslant \alpha^{n} p_0,
$$
$$
\cdots .
$$
By the triangle inequality,
$$
d(x_n,\,x_{n+p})\leqslant d(x_{n},\,x_{n+1})+d(x_{n+1},\,x_{n+2})+\ldots+d(x_{n+p-1},\,x_{n+p})
$$
$$
\leqslant \alpha^{n-1}p_0+\alpha^{n}p_0+\cdots +\alpha^{n+p-2}p_0 = \alpha^{n-1}(1+\alpha+\ldots+\alpha^{p-1})p_0
=\alpha^{n-1}\frac{1-\alpha^{p}}{1-\alpha}p_0.
$$
Since by the supposition $0\leqslant\alpha<1$, then $d(x_n,\,x_{n+p})<\alpha^{n-1}\frac{1}{1-\alpha}p_0$. Hence, $d(x_n,\,x_{n+p})\to 0$ as $n\to \infty$ for every $p>0$. Thus, $\{x_n\}$ is a Cauchy sequence. By completeness of $(X,d)$, this sequence has a limit $x^*\in X$.

Let us prove that
$Tx^*=x^*$.
By the triangle inequality and by inequality~(\ref{e1}) we have
$$
d(x^*,Tx^*)\leqslant d(x^*,x_{n})+d(x_{n},Tx^*)
=d(x^*,x_{n})+d(Tx_{n-1},Tx^*)
$$
$$
\leqslant d(x^*,x_{n})+d(Tx_{n-1},Tx^*)+d(Tx_{n-1},Tx_{n})+d(Tx_{n},Tx^*)
$$
$$
\leqslant d(x^*,x_{n})+\alpha(d(x_{n-1},x^*)+d(x_{n-1},x_{n})+d(x_{n},x^*)).
$$
Since all the terms in the previous sum tend to zero as $n\to \infty$, we obtain $d(x^*,Tx^*)=0$.

Suppose that there exists at least three pairwise distinct fixed points $x$, $y$ and $z$.  Then $Tx=x$, $Ty=y$ and $Tz=z$, which contradicts to~(\ref{e1}).
\end{proof}

\begin{rem}
Suppose that under the supposition of the theorem the mapping $T$ has a fixed point $x^*$ which is a limit of some iteration sequence $x_0, x_1=Tx_0, x_2=Tx_1,\ldots$ such that $x_n\neq x^*$ for all $n=1,2,\ldots$. Then $x^*$ is a unique fixed point.
Indeed, suppose that $T$ has another fixed point $x^{**}\neq x^*$.
It is clear that $x_n\neq x^{**}$ for all $n=1,2,\ldots$. Hence, we have that the points $x^*$, $x^{**}$ and $x_n$ are pairwise distinct for all $n=1,2,\ldots$. Consider the ratio
$$
R_n=\frac{d(Tx^*,Tx^{**})+d(Tx^*,Tx_{n})+d(Tx^{**},Tx_{n})}{d(x^*,x^{**})+d(x^*,x_{n})+d(x^{**},x_{n})}
$$
$$
=\frac{d(x^*,x^{**})+d(x^*,x_{n+1})+d(x^{**},x_{n+1})}{d(x^*,x^{**})+d(x^*,x_{n})+d(x^{**},x_{n})}.
$$
Taking into consideration that $d(x^*,x_{n+1})\to 0$, $d(x^*,x_{n})\to 0$, $d(x^{**},x_{n+1})\to d(x^{**},x^*)$ and $d(x^{**},x_{n})\to d(x^{**},x^*)$, we obtain
$R_n\to 1$ as $n\to \infty$, which contradicts to condition~(\ref{e1}).
\end{rem}

\begin{ex}
Let us construct an example of the mapping $T$ contracting perimeters of triangles which has exactly two fixed points.
Let $X=\{x,y,z\}$, $d(x,y)=d(y,z)=d(x,z)=1$, and let $T\colon X\to X$ be such that $Tx= x$, $Ty= y$ and $Tz= x$. One can easily see that conditions (i) and (ii) of Theorem~\ref{t1} are fulfilled.
\end{ex}

\begin{ex}
Let us show that condition (i) of Theorem~\ref{t1} is necessary.
Let $X=\{x,y,z\}$, $d(x,y)=d(y,z)=d(x,z)=1$, and let $T\colon X\to X$ be such that $Tx=y$, $Ty=x$ and $Tz=x$. One can easily see that condition (ii) of Theorem~\ref{t1} is fulfilled but $T$ does not have any fixed point.
\end{ex}

Let $(X,d)$ be a metric space. Then a mapping $T\colon X\to X$ is called a \emph{contraction mapping} on $X$ if there exists $\alpha\in [0,1)$ such that
\begin{equation}\label{e3}
d(Tx,Ty)\leqslant \alpha d(x,y)
\end{equation}
for all $x,y \in X$.

\begin{cor}
\textbf{(Banach fixed-point theorem)}
Let $(X,d)$ be a  non\-empty complete metric space with a contraction mapping $T\colon X\to X$.  Then $T$ admits a unique fixed point.
\end{cor}

\begin{proof}
For $|X|=1,2$ the proof is trivial. Let $|X|\geqslant 3$. Suppose that there exists $x\in X$ such that $T(Tx)=x$. Consequently, $d(x,Tx)=d(Tx,x)=d(Tx,T(Tx))$, which contradicts to~(\ref{e3}). Thus, condition (i) of Theorem~\ref{t1} holds. Let $x, y, z\in X$ be pairwise distinct. By~(\ref{e3}) we obtain $d(T(x),T(y))\leqslant \alpha d(x,y)$, $d(T(y),T(z))\leqslant \alpha d(y,z)$ and $d(T(x),T(z))\leqslant \alpha d(x,z)$  which immediately implies condition (ii) of Theorem~\ref{t1}. This completes the proof of existence of fixed point.

The uniqueness can be shown in a standard way.
\end{proof}

Let $(X,d)$ be a metric space and let $x,y,z \in X$. We shall say that the point $y$ \emph{lies between} $x$ and $z$ in the metric space $(X,d)$ if the extremal version of the triangle inequality
\begin{equation}\label{e4}
d(x,z)=d(x,y)+d(y,z)
\end{equation}
holds.

\begin{ex}
Let us construct an example of a mapping $T\colon X\to X$ contracting perimeters of triangles that is not a contraction mapping for a metric space $X$ with $|X|=\aleph_0$. Let $X=\{x^*, x_0,x_1,\ldots \}$ and let $a$ be positive real number.
Define a metric $d$ on $X$ as follows:
$$
d(x,y)=
\begin{cases}
a/2^{\lfloor i /2 \rfloor}, &\text{if } x=x_{i},  y=x_{i+1}, i=0,1,2,...,\\
d(x_i,x_{i+1})+\cdots+d(x_{j-1},x_j), &\text{if } x=x_i, \,  y=x_{j}, \, i+1<j,\\
4a-d(x_0,x_i), &\text{if } x=x_i, \, y=x^*,\\
0, &\text{if } x=y,
\end{cases}
$$
where $\lfloor\cdot\rfloor$ is the floor function.
\begin{figure}[ht]
\begin{center}
\begin{tikzpicture}[scale=0.5]
\draw (1,0) node [above] {\small{$a$}};
\draw (3,0) node [above] {\small{$a$}};
\draw (5,0) node [above] {\small{$\frac{a}{2}$}};
\draw (7,0) node [above] {\small{$\frac{a}{2}$}};
\draw (9,0) node [above] {\small{$\frac{a}{4}$}};
\draw (11,0) node [above] {\small{$\frac{a}{4}$}};

\draw (0,0) node [below] {$x_0$} --
      (2,0) node [below] {$x_1$} --
      (4,0) node [below] {$x_2$} --
      (6,0) node [below] {$x_3$} --
      (8,0) node [below] {$x_4$} --
      (10,0) node [below] {$x_5$} --
      (12,0) node [below] {$x_6$};
\draw (14,0) node [below] {$x^*$};

 \foreach \i in {(0,0),(2,0),(4,0),(6,0),(8,0),(10,0),(12,0),(14,0)}
  \fill[black] \i circle (3pt);

 \foreach \i in {(12+0.3,0),(12+0.6,0),(12+0.9,0),(12+1.2,0),(12+1.5,0),(12+1.8,0)}
  \fill[black] \i circle (1pt);
\end{tikzpicture}
\begin{center}
\caption{The points of the space $(X,d)$ with consecutive distances between them.}\label{fig1}
\end{center}
\end{center}
\end{figure}
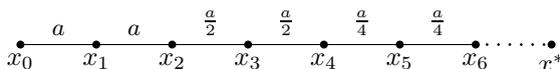

The reader can easily verify that for every three different points from the set $X$ one of them lies between the two others, see Figure~\ref{fig1}. Moreover, the space is complete with the single accumulation point $x^*$.

Define a mapping  $T\colon X\to X$ as $Tx_i=x_{i+1}$ for all $i=0,1,\ldots$ and $Tx^*=x^*$. Since
$d(x_{2n},x_{2n+1})=d(T{x_{2n}},T{x_{2n+1}})$, $n=0,1,2...$, using~(\ref{e3}) we see that $T$ is not a contraction mapping.

Let us show that inequality~(\ref{e1}) holds for every three pairwise distinct points from the space $X$.
Consider first triplets of points $x_i, x_j, x^* \in X$ with $0\leqslant i<j$.
According to the definition of the metric $d$ we have
$$
d(x_i,x_j)+d(x_j,x^*)+d(x_i,x^*)=2d(x_i,x^*)=8a-2d(x_0,x_i)
$$
and
$$
d(Tx_i,Tx_j)+d(Tx_j,Tx^*)+d(Tx_i,Tx^*)=2d(Tx_i,Tx^*)=8a-2d(x_0,x_{i+1}).
$$
According to the formula for a geometric series that computes the sum of $n$ terms we have
$$
d(x_0,x_i)=
\begin{cases}
4a(1-(1/2)^n), &\text{if } i=2n,\\
4a(1-(1/2)^n)-a/{2^{n-1}}, &\text{if } i=2n-1,\\
\end{cases}
$$
$n=1,2,\ldots$.
Note also that $d(x_0,x_{i+1})=d(x_0,x_i)+a/(2^{\lfloor i/2 \rfloor})$. Consider the ratio
$$
\frac{d(Tx_i,Tx_j)+d(Tx_j,Tx^*)+d(Tx_i,Tx^*)}{d(x_i,x_j)+d(x_j,x^*)+d(x_i,x^*)}
=\frac{8a-2d(x_0,x_{i+1})}{8a-2d(x_0,x_i)}
$$
$$
=\frac{4a-d(x_0,x_{i})-a/(2^{\lfloor i/2 \rfloor})}{4a-d(x_0,x_i)}
$$
$$
=
\begin{cases}
\frac{4a-4a(1-(1/2)^n)-a/(2^{\lfloor i/2 \rfloor})}{4a-4a(1-(1/2)^n)}, &\text{if } i=2n,\\
\frac{4a-4a(1-(1/2)^n)+a/2^{n-1}-a/(2^{\lfloor i/2 \rfloor})}{4a-4a(1-(1/2)^n)+a/2^{n-1}}, &\text{if } i=2n-1,
\end{cases}
$$
$$
=
\begin{cases}
\frac{3}{4}, &\text{if } i=2n,\\
\frac{2}{3}, &\text{if } i=2n-1.
\end{cases}
$$

Let now $x_i,x_j,x_k \in X$ be such that $0\leqslant i<j<k$. Using Figure~\ref{fig1}, we see that
$$d(x_i,x_j)+d(x_j,x_k)+d(x_i,x_k)-(d(Tx_i,Tx_j)+d(Tx_j,Tx_k)+d(Tx_i,Tx_k))
$$
$$
=2(a/2^{\lfloor i/2 \rfloor}-a/2^{\lfloor k/2 \rfloor}).
$$
Consider the ratio
$$
R_{i,k}=\frac{d(Tx_i,Tx_j)+d(Tx_j,Tx_k)+d(Tx_i,Tx_k)}
{d(x_i,x_j)+d(x_j,x_k)+d(x_i,x_k)}
$$
$$
=\frac{d(x_i,x_j)+d(x_j,x_k)+d(x_i,x_k)-2(a/2^{\lfloor i/2 \rfloor}-a/2^{\lfloor k/2 \rfloor})}
{d(x_i,x_j)+d(x_j,x_k)+d(x_i,x_k)}
$$
$$
=1-2\frac{(a/2^{\lfloor i/2 \rfloor}-a/2^{\lfloor k/2 \rfloor})}
{d(x_i,x_j)+d(x_j,x_k)+d(x_i,x_k)}.
$$
Observe that $i+1<k$. Hence,
\begin{equation}\label{e5}
a/2^{\lfloor k/2 \rfloor} \leqslant a/(2\cdot2^{\lfloor i/2 \rfloor}).
\end{equation}
Using the structure of the space $(X,d)$, one can show that $d(x_i,x^*)\leqslant 4d(x_i,x_{i+1})$. Clearly, $d(x_i,x_k)\leqslant d(x_i,x^*)$. Hence, $d(x_i,x_k)\leqslant 4d(x_i,x_{i+1})$. From  equality~(\ref{e4}) and the last inequality it follows that
$$
d(x_i,x_j)+d(x_j,x_k)+d(x_i,x_k)=2d(x_i,x_k)\leqslant 8d(x_i,x_{i+1})=8a/(2^{\lfloor i/2\rfloor}).
$$
Using this inequality and inequality~(\ref{e5}) we obtain
$$
R_{i,k}\leqslant 1-2\frac{(a/2^{\lfloor i/2 \rfloor}-a/(2\cdot2^{\lfloor i/2 \rfloor}))}
{8a/2^{\lfloor i/2 \rfloor}}=\frac{7}{8}.
$$
Hence, inequality~(\ref{e1}) holds for every three pairwise distinct points from the space $X$ with the coefficient $\alpha=\frac{7}{8}=\max\{\frac{2}{3},\frac{3}{4},\frac{7}{8}\}$.
\end{ex}

\textbf{Acknowledgements.} The author was partially supported by the Grant EFDS-FL2-08 of The European Federation of Academies of Sciences and Humanities (ALLEA).

% ------------------------------------------------------------------------
\end{document}